\documentclass[11pt,reqno,tbtags,a4paper]{amsart}
\usepackage{amssymb}
\usepackage{url}
\usepackage[square,numbers]{natbib}
\bibpunct[, ]{[}{]}{;}{n}{,}{,}

\title
{A note on Polya urns: the winner may lead all the time }

\date{April 1, 2016; revised 4 April, 2016.\\
This note was posted on my web site in 2016; since there has been
interest to refer to it, it is now also posted on arXiv in 2025.
}

\author{Svante Janson}
\thanks{Partly supported by the Knut and Alice Wallenberg Foundation}
\address{Department of Mathematics, Uppsala University, PO Box 480,
SE-751~06 Uppsala, Sweden}
\email{svante.janson@math.uu.se}
\urladdr{http://www.math.uu.se/svante-janson}

\subjclass[2010]{} 

\renewcommand\le{\leqslant}
\renewcommand\ge{\geqslant}

\allowdisplaybreaks

\theoremstyle{plain}
\newtheorem{theorem}{Theorem}

\theoremstyle{definition}

\newtheorem{remark}[theorem]{Remark}

\theoremstyle{remark}

\newenvironment{romenumerate}[1][-10pt]{
\addtolength{\leftmargini}{#1}\begin{enumerate}
 }{\end{enumerate}}

\newcounter{oldenumi}
{\setcounter{oldenumi}{\value{enumi}}
\begin{romenumerate} \setcounter{enumi}{\value{oldenumi}}}
{\end{romenumerate}}

\newcounter{thmenumerate}

\newcounter{xenumerate}   

\newcommand\marginal[1]{\marginpar[\raggedleft\tiny #1]{\raggedright\tiny#1}}
\newcommand\REM[1]{{\raggedright\texttt{[#1]}\par\marginal{XXX}}}

\begingroup
  \divide\count255 by 60
  \multiply\count255 by -60
  \advance\count255 by \time
  \ifnum \count255 < 10 \xdef\klockan{\the\count1.0\the\count255}
  \else\xdef\klockan{\the\count1.\the\count255}\fi
\endgroup

\newcommand\set[1]{\ensuremath{\{#1\}}}

\newcommand\bigpar[1]{\bigl(#1\bigr)}

\def\rompar(#1){\textup(#1\textup)}    
\newcommand\xfrac[2]{#1/#2}

\def\xexp(#1){e^{#1}}

\newcommand\ntoo{\ensuremath{{n\to\infty}}}

\newcommand\ttoo{\ensuremath{{t\to\infty}}}

\newcommand\punkt{.\spacefactor=1000}    
    
\newcommand\ie{i.e\punkt}
\newcommand\eg{e.g\punkt}

\newcommand{\as}{a.s\punkt}

\newcounter{CC}
\newcounter{cc}

\newcommand\PP{\operatorname{\mathbb P{}}}

\renewcommand\phi{\xxx}  

\newcommand\xb{\mathsf b}
\newcommand\xw{\mathsf w}
\newcommand\xB{\overline B}
\newcommand\xW{\overline W}

\newcommand{\Polya}{P\'olya}

\begin{document}

\maketitle

We consider a \Polya{} urn where, as in \Polya's original urn
\cite{EggPol,Polya1931},
a ball always is replaced together with additional balls of the same colour
only; however, we allow  the urn to be unfair in the sense that the number
of additional balls may depend on the colour. 
To be precise, the urn contains balls of $q$ colours; we draw a ball
(uniformly at random) and if the drawn ball has colour $i$, we replace it
together with $m_i$ additional balls of colour $i$, for some given (fixed) 
integers $m_i>0$. 
We assume that the urn start with some set of balls, containing at
least one ball of each colour.

It is well-known, and easy to see
by the argument below, that if there is a colour $i$ such that $m_i>m_j$ for
all colours $j\neq i$, then the proportion of balls of colour $i$ tends to 1
a.s.
Of course, even if colour 1, say, wins eventually, it is possible that
colour 2 is lucky initially, and that at some time $n$, most balls are
of colour 2.
The purpose of this note is to show that this happens with probability
strictly less than 1, provided it does not occur already in the initial
position. 
Moreover, this extends to the case when some or all of the numbers $m_i$ are
equal, say with $m_1=m_2\ge m_3\ge\dots$; in 
this case, the proportion of balls of colour 1 converges to some random limit
in $(0,1)$, and the probability that colour 1 eventually dominates lies
strictly between 0 and 1; we show that there is also in this case a positive
probability that colour 1 dominates at all times, if it does so initially.

To be precise, we have  the following
theorem, where we for simplicity consider two colours only.
(The proof is simple and this has presumably been observed before, but we do
not know a reference and therefore give a complete proof.)

\begin{theorem}
Consider a \Polya{} urn with balls of two colours, black and white, 
with the replacement rule that if a ball of colour $i\in\set{\xb,\xw}$ is
drawn, it is replaced together with $m_i$ 
additional balls of the same colour, for some given positive integers
$m_\xb,m_\xw$. 
Let $B_n$ and $W_n$ be the numbers of black and white balls after $n$ draws,
with initial conditions $B_0=b_0$ and $W_0=w_0$ for some given $b_0,w_0$.
Suppose that $m_\xb\ge m_\xw$ and that $b_0>w_0\ge0$.
Then there is a positive probability that $B_n>W_n$ for all $n\ge0$.
\end{theorem}

\begin{proof}
  We use the standard method of embedding the urn process into a continuous
  time Markov branching process, see \eg{} \cite[Section V.9]{AN}; we give
  each ball an exponential clock, that rings after a random time with the
  distribution Exp(1). When the clock rings at a ball of colour $i$, 
the ball is replaced by $m_i+1$ new  balls of the same colour $i$, each with
a new clock. (All clocks are independent.)
Then the original urn process is  the same (\ie, has the same distribution)
as this continuous time process observed at the times some clock rings.
More formally, 
let $\xB_t$ and $\xW_t$ be the number of black and white balls at time $t\ge0$,
and let $\tau_n$ be the time the $n$-th clock rings. Then, the process
$(\xB_{\tau_n},\xW_{\tau_n})_{n\ge0}$ has the same distribution as the \Polya{}
  urn process, and we may assume $B_n=\xB_{\tau_n}$, $W_n=\xW_{\tau_n}$.

In the continuous time process, the black and white balls act independently,
so $\xB_t$ and $\xW_t$ are two
independent continuous-time branching processes.
Hence, see \cite[Theorems III.7.1 and III.7.2]{AN}, 
there exist independent positive
random variables $Y_\xb,Y_\xw$ such that a.s.
\begin{align}
  e^{-m_\xb t}\xB_t &\to Y_\xb>0,
\\
  e^{-m_\xw t}\xW_t &\to Y_\xw>0.
\end{align}
Consequently,  as \ttoo, \as,
\begin{equation}
    e^{(m_\xb-m_\xw) t}\frac{\xW_t}{\xB_t} \to \frac{Y_\xw}{Y_\xb}<\infty,
\end{equation}
and thus
\begin{equation}\label{win}
   \frac{\xW_t}{\xB_t} \to 
Z:=
\begin{cases}
  0,& m_\xb>m_\xw,
\\
\xfrac{Y_\xw}{Y_\xb}, & m_\xb=m_\xw.
\end{cases}
\end{equation}
As a consequence,  as \ntoo, \as,
\begin{equation}\label{eleo}
   \frac{W_n}{B_n} =
   \frac{\xW_{\tau_n}}{\xB_{\tau_n}} \to Z.
\end{equation}
Furthermore, both $Y_\xb$ and $Y_\xw$ have support on the entire positive
half-axis. Consequently, $\PP(Y_\xw<Y_\xb)>0$ and thus
\eqref{win} implies
$\PP(Z<1)>0$ for any $m_\xb$ and $m_\xw$. 
Hence, \eqref{eleo} implies
\begin{equation}
  \PP\bigpar{\limsup_\ntoo W_n/B_n <1} \ge \PP(Z<1)>0.
\end{equation}
It follows that there exists an integer $N$ such that
\begin{equation}\label{ql}
\PP\bigpar{W_n/B_n<1 \text{ for all $n\ge N$}}>0.  
\end{equation}
There is only a finite number of possible outcomes of $(B_N,W_N)$, and
consequently \eqref{ql} implies that there are integers $(b_N,w_N)$ such
that
\begin{equation}\label{qm}
\PP\bigpar{(B_N,W_N)=(b_N,w_N) \text{ and } W_n/B_n<1 \text{ for all $n\ge N$}}
>0
\end{equation}
and consequently,
\begin{align}
\PP&\bigpar{(B_N,W_N)=(b_N,w_N)}>0, \label{qa}
\intertext{and}
\PP&\bigpar{ W_n/B_n<1 \text{ for all $n\ge N$}\mid (B_N,W_N)=(b_N,w_N)}>0.
\label{qb}
\end{align}
By \eqref{qa}, it is possible to reach $(b_N,w_N)$ by some sequence of $N$
draws, starting at $(b_0,w_0)$; thus $b_N=b_0+k_\xb m_\xb$
and $w_N=w_0+k_\xw m_\xw$, where $k_\xb+k_\xw=N$.

Consider now the event that
we draw a black ball in the first $k_\xb$
draws, and then a white ball in the following $k_\xw$ draws;
this evidently has  positive probability. 
Furthermore, then $(B_N,W_N)=(b_N,w_N)$. Moreover, in this case, 
for $0\le n\le N$,
$B_n-W_n$ first
increases and then decreases,  and since $B_0-W_0=b_0-w_0>0$
and $B_N-W_N=b_N-w_N>0$, we have $B_n-W_n>0$ for every $n\le N$.

Since the urn process is a Markov process, it is by \eqref{qb} possible to
continue, with positive probability, so that $B_n>W_n$ also for all $n\ge
N$.
Hence, with positive probability, $B_n>W_n$ for all $n\ge0$.
\end{proof}

\begin{remark}
  The theorem, and its proof, readily extends to any (finite) number of
  colours. Several versions are possible.
For example, if we assume that $m_1\ge m_j$ for every colour $j\neq1$, and a
  majority of the balls at time 0 are of colour 1, then there is a positive
  probability that a majority has colour 1 at every time $n$.
Alternatively,
still assuming that $m_1\ge m_j$ for every colour $j\neq1$, if a
plurality (relative majority) of the balls at time 0 are of colour 1, 
\ie, the number of balls of colour 1 is larger than the number of any other
given colour, then there is a positive
  probability that a plurality has colour 1 at every time $n$.
We leave the details, and further variations, to the reader.
\end{remark}

\newcommand\AAP{\emph{Adv. Appl. Probab.} }
\newcommand\JAP{\emph{J. Appl. Probab.} }
\newcommand\JAMS{\emph{J. \AMS} }
\newcommand\MAMS{\emph{Memoirs \AMS} }
\newcommand\PAMS{\emph{Proc. \AMS} }
\newcommand\TAMS{\emph{Trans. \AMS} }
\newcommand\AnnMS{\emph{Ann. Math. Statist.} }
\newcommand\AnnPr{\emph{Ann. Probab.} }
\newcommand\CPC{\emph{Combin. Probab. Comput.} }
\newcommand\JMAA{\emph{J. Math. Anal. Appl.} }
\newcommand\RSA{\emph{Random Struct. Alg.} }
\newcommand\ZW{\emph{Z. Wahrsch. Verw. Gebiete} }
\newcommand\DMTCS{\jour{Discr. Math. Theor. Comput. Sci.} }

\newcommand\AMS{Amer. Math. Soc.}
\newcommand\Springer{Springer-Verlag}
\newcommand\Wiley{Wiley}

\newcommand\vol{\textbf}
\newcommand\jour{\emph}
\newcommand\book{\emph}
\newcommand\inbook{\emph}
\def\no#1#2,{\unskip#2, no. #1,} 
\newcommand\toappear{\unskip, to appear}

\newcommand\arxiv[1]{\texttt{arXiv:#1}}
\newcommand\arXiv{\arxiv}

\def\nobibitem#1\par{}

\end{document}